\documentclass[12pt, leqno, letterpaper]{article}

\usepackage{amsmath,amssymb, amsthm}
\usepackage[english]{babel}
\usepackage{subfig}
\usepackage{graphicx}
\usepackage{enumerate,color,xcolor}
\usepackage{hyperref}

\graphicspath{{code/}}
\newtheorem{theorem}{Theorem}[section]

\newtheorem{lemma}{Lemma}[section]

\theoremstyle{remark}

\theoremstyle{definition}

\newcommand{\footremember}[2]{%
    \footnote{#2}
    \newcounter{#1}
    \setcounter{#1}{\value{footnote}}%
}

\newtheorem{example}{Example}
\providecommand{\keywords}[1]{\textbf{\textit{key words:}} #1}

\title{Limit theorems for a discrete-time marked Hawkes process}
\author{Haixu Wang \footremember{hw}{Department of Mathematics, Florida State University, Tallahassee, FL 32306 \newline Email: hwang@math.fsu.edu}}

\date{}

\begin{document}

\maketitle

\begin{abstract}
	Hawkes process is a self-exciting point process with wide applications in many fields, such as finance, seismology, and ecology. Hawkes processes are defined for continuous time setting. However, data is often recorded in a discrete-time or aggregated scheme. Thus, in order to model the temporal process in aggregated way, oftentimes a discrete-time type Hawkes process is more desirable. In this paper, we study the limit theorems for a discrete-time marked Hawkes process.
\end{abstract}
\small \keywords{discrete-time; self-exciting; marked Hawkes process; univariate; limit theorems; law of large numbers; central limit theorem}

\section{Introduction}

Hawkes process is a self-exciting simple point process named after Hawkes \cite{Hawkes}. In contrast to a standard Poisson point process, the intensity of Hawkes process depends on its entire history, which can model the self-exciting, clustering effect. In applications, Hawkes process can model temporal stochastic arrivals of events that evolve in continuous time. For instance, Hawkes processes were used to model financial point process like trading orders process in high-frequency trading \cite{BDM,BarM,BauH}. Hawkes process is mathematically
tractable and widely used in practice, especially the linear Hawkes process. The applications can be also found in seismology, neuroscience, social network, etc. For a list of references for these applications, see \cite{ZhuTh,Liniger}. Additionally, it is worth to mention that in general, Hawkes process is not a Markov process unless the exciting function is in exponential form (see for instance \cite{ZhuIV}). 

In general, based on the intensity, Hawkes process can be classified as linear and nonlinear Hawkes processes. For the linear model, there were extensive studies in the stability, law of large numbers, central limit theorem, large deviations, Bartlett spectrum, etc. For a survey on linear Hawkes processes and related self-exciting processes, Poisson cluster processes, affine point processes, etc., see \cite{DV}.  A nonlinear Hawkes process
was first introduced by \cite{BM}. 
Due to the lack of immigration-birth representation and computational tractability, nonlinear Hawkes processes are much less studied. However, there were some efforts in this direction.
The central limit theorem, the large deviation principles for nonlinear Hawkes processes can be found in \cite{ZhuI,ZhuIV}. The Hawkes process also can be extended to the multivariate setting. For theories and applications of multivariate Hawkes processes, we refer to \cite{Liniger}. 

Hawkes processes model the unevenly spaced self-exciting arrivals of events in time. However, time series data is also collected on a fixed phase. Modeling the evenly spaced arrivals of events in time and capturing the self-exciting character require a discrete type model. 
Compared with Hawkes processes, discrete-time Hawkes processes are not well studied yet. \cite{XZW2020} for the first time proposed a discrete-time self-exciting and mutually-exciting model to model the deposit and withdrawal behaviors of money market accounts. 
More recently, the discrete-time self-exciting model was also applied to study the infection and death of COVID-19 in \cite{BDKVJ2021}. In a related work, \cite{Seol} proposed a discrete-time Hawkes-like model with 0-1 arrivals and studied its limit theorems. 
In this paper, we study the univariate marked discrete-time self-exciting process (discrete-time Hawkes model) which is a special case of the model proposed in \cite{XZW2020} and derive the limit theorems for our model.

Next, we will briefly review the existing limit theorems of a marked linear Hawkes model and an unmarked discrete-time Hawkes model.

\subsection{Limit theorems for marked Hawkes processes}
Let $N$ be a simple point process on $\mathbb{R}$ and let $\mathcal{F}^{-\infty}_t:=\sigma\left(N(C),C\in\mathcal{B}\left(\mathbb{R}\right),C\in(-\infty,t]\right)$ be an increasing family of $\sigma$-algebras. Any non-negative $\mathcal{F}^{-\infty}_t$-progressively measurable process $\lambda_t$ with 
$$\mathbb{E}\left[N(a,b]|\mathcal{F}^{-\infty}_a\right]=\mathbb{E}\left[\int_{a}^{b}\lambda_s ds|\mathcal{F}^{-\infty}_a\right]$$
a.s. for all interval $(a,b]$ is called the $\mathcal{F}^{-\infty}_t$-intensity of N. $N_t:=N(0,t]$ denotes the number of points in the interval $(0,t]$. A marked univariate linear Hawkes process is a simple point process whose intensity is defined as
\begin{equation}
\label{lli}
\lambda_t := \nu +\int_{(-\infty,t)\times\mathbb{X}}h(t-s,\ell)N(ds,d\ell)
\end{equation}
where $\nu>0$, and random marks belong to a measurable space $\mathbb{X}$ with common law $q(d\ell)$. Let $h(\cdot,\cdot):\mathbb{R}^{+}\times\mathbb{X}\to\mathbb{R}^{+}$ is integrable, and $||h||_{L^1} = \int_{0}^{\infty}\int_{\mathbb{X}}h(t,\ell)q(d\ell)dt<\infty$. The integral in equation $\left(\ref{lli}\right)$ stands for $\int_{(-\infty,t)\times \mathbb{X}}h(t-s,\ell)N(ds,d\ell)=\sum_{\tau_i<t}h(t-\tau_i,\ell_i)$, where $(\tau_i)_{i\ge1}$ are the occurrences of the points before time t, and the $(\ell_i)_{i\ge1}$ are i.i.d. random marks, $\ell_i$ being independent of previous arrival times $\tau_j$, $j\le i$. Let $H(\ell):=\int_{0}^{\infty}h(t,
\ell)dt$ for any $\ell\in \mathbb{X}$. Assume that
\begin{equation}
\label{Ha}
\int_{\mathbb{X}}H(\ell)q(d\ell)<1
\end{equation}
Under assumption $\left(\ref{Ha}\right)$, there exists a unique stationary version of the linear marked Hawkes process satisfying the dynamic formula of intensity and by ergodic theorem, a law of large numbers holds
$$\lim\limits_{t\to\infty}\frac{N_t}{t}=\frac{\nu}{1-\mathbb{E}^q[H(\ell)]},$$
see \cite{DV}. \cite{Karabash} obtained central limit theorem follows. \cite{Karabash} showed that assume $\lim\limits_{t\to\infty}t^{\frac{1}{2}}\int_{t}^{\infty}\mathbb{E}^q[h(s, a)]ds = 0$ and that assumption $\left(\ref{Ha}\right)$ holds. Then,
\begin{equation}
\label{clt_cmark}
\frac{N_t-\frac{\nu t}{1-\mathbb{E}^q[H(\alpha)]}}{\sqrt{t}} \to N\left(0,\frac{\nu\left(1+Var^q[H(\alpha)]\right)}{(1-\mathbb{E}^q[H(\alpha)])^3}\right)
\end{equation}
in distribution as $t\to \infty$.

\subsection{Limit theorems for 0-1 discrete unmarked Hawkes processes}
\cite{Seol} proposed a 0-1 discrete Hawkes processes as follows.
Let $\left(X_n\right)^{\infty}_{n=1}$ be a sequence taking values on $\{0,1\}$ defined as follows.
\begin{itemize}
	\item $X_1=1$ with probability $\alpha_0$ and $X_1 = 0$ otherwise.\\
	\item Conditional on $X_1,X_2,...,X_{n-1}$, we have $X_n = 1$ with probability.
	$$\alpha_0 + \sum_{i=1}^{n-1}\alpha_{n-i}X_i,$$
	and $X_n = 0$ otherwise.
\end{itemize}
Let $\hat{\mathbb{N}}=\mathbb{N}\bigcup\{0\}$. 
We assumed that for $i\in \mathbb{N}$, $\alpha_i > 0$ is a given sequence of positive numbers and $\sum_{i=0}^{\infty}\alpha_i<1$.

Then define $S_n := \sum_{i=1}^{n}X_i$. There is law of large number theorem, i.e.
$$\frac{S_n}{n} \to \mu := \frac{\alpha_0}{1-\sum_{i=1}^{\infty}\alpha_i}$$
in probability as $n\to\infty$. 
And additionally, with assumption $\sqrt{n}\sum_{i=n}^{\infty}\alpha_i\to0$ as $n\to\infty$ and $\frac{1}{\sqrt{n}}\sum_{i=1}^{n}i\alpha_i\to0$ as $n\to\infty$, the central limit theorem follows
$$\frac{S_n-\mu n}{\sqrt{n}} \to N\left(0,\frac{\mu(1-\mu)}{\left(1-\sum_{j=1}^{\infty}\alpha_j\right)^2}\right)$$ 
in distribution as $n\to\infty$. 

\textbf{The other related Literature.}
The existing literature on limit theorems of Hawkes processes mainly focus on linear Hawkes models. \cite{Bacry} showed the functional law of large numbers and functional central limit theorems for multivariate Hawkes processes with large time asymptotics setting. For nearly unstable Hawkes process, the limit theorems were dervied by \cite{JR}. Recently, \cite{GZ} established limit theorems and large deviation for linear Markovian Hawkes processes with a large initial intensity. \cite{Horst} established a functional law of large numbers and a functional central limit theorem for marked Hawkes point measures with homogeneous immigration. 

\textbf{Organization of this paper.}
The rest of the paper is organized as follows. In section 2, we state our main results. The proof of the main results can be found in section 3.
\section{Main Results}
In this section, we will describe our main results. 
We refer the model in this paper as a univariate discrete-time marked Hawkes process.
We define the model analogous to Hawkes process as follows.
Define $\hat{\mathbb{N}}=\mathbb{N}\bigcup\{0\}$, $X_{0}=N_{0}=0$.
In other words, the Hawkes process has no memory of unrecorded history.
Let $t\in\hat{\mathbb{N}}$,  $\alpha(t):\hat{\mathbb{N}}\rightarrow\mathbb{R}_{+}$
be a positive function on $\hat{\mathbb{N}}$. It is worth to mention that $\alpha(\cdot)$ is an exponential function in \cite{XZW2020}.
Therefore, in our paper, we extended the univariate case considered in \cite{XZW2020}.
We define $\Vert\alpha\Vert_{1}:=\sum_{t=1}^{\infty}\alpha(t)$
as the $\ell_{1}$ norm of $\alpha$. 
Conditional on $X_{t-1},X_{t-2},\ldots,X_{1}$, 
we define $Z_{t}$ as a Poisson random variable
with mean
\begin{equation}
\label{dintensity}
\lambda_{t}:=\nu+\sum_{s=1}^{t-1}\alpha(s)X_{t-s},
\end{equation}
and define
\begin{equation}
\label{drv}
X_{t}=\sum_{j=1}^{Z_{t}}\ell_{t,j},
\end{equation}
where $\ell_{t,j}$ are positive random variables
that are i.i.d. in both $t$ and $j$ with finite mean and the probability measure $\mathbb{P}$.
Throughout the paper, we assume that $\Vert\alpha\Vert_{1}\mathbb{E}[\ell_{1,1}]<1$.
Finally, we define $N_{t}:=\sum_{s=1}^{t}Z_{s}$
and $L_{t}:=\sum_{s=1}^{t}X_{s}$. 

We obtain the law of large numbers for $N_t$ and $L_t$ as follows. 
\begin{theorem}[Law of Large Numbers for $N_t$]\label{LLN1}
For $N_t:=\sum_{s=1}^{t}Z_{s}$ and $Z_t$ is defined by equation~\eqref{dintensity}, we can find
\begin{equation}
\lim_{t\rightarrow\infty}\frac{N_{t}}{t}=\frac{\nu}{1-\Vert\alpha\Vert_{1}\mathbb{E}[\ell_{1,1}]},
\end{equation}
in probability as $t\rightarrow\infty$.
\end{theorem}

\begin{theorem}[Law of Large Numbers for $L_t$]\label{LLN2}
For $N_t:=\sum_{s=1}^{t}X_{s}$ and $X_t$ is defined by equation~\eqref{drv}, we can show
	\begin{equation}
	\lim_{t\rightarrow\infty}\frac{L_{t}}{t}=\frac{\nu\mathbb{E}[\ell_{1,1}]}{1-\Vert\alpha\Vert_{1}\mathbb{E}[\ell_{1,1}]},
	\end{equation}
	in probability as $t\rightarrow\infty$.
\end{theorem}
Next, we present the central limit theorems.
\begin{theorem}[Central Limit Theorem]\label{CLT1}
	Assume that the first four moments of $\ell_{t,i}$ are finite and $\lim_{t\rightarrow\infty}\frac{1}{\sqrt{t}}\sum_{u=1}^{t-1}\sum_{s=1+u}^{\infty}\alpha(s)=0$.
	\begin{equation}
	\frac{1}{\sqrt{t}}
	\left(N_{t}-\frac{\nu t}{1-\Vert\alpha\Vert_{1}\mathbb{E}[\ell_{1,1}]}\right)
	\rightarrow N\left(0,\frac{\nu(1+\Vert\alpha\Vert_{1}^{2}\text{Var}(\ell_{1,1}))}{(1-\Vert\alpha\Vert_{1}\mathbb{E}[\ell_{1,1}])^{3}}\right),
	\end{equation}
	in distribution as $t\rightarrow\infty$.
\end{theorem}

\begin{theorem}[Central Limit Theorem]\label{CLT2}
Assume that the first four moments of $\ell_{t,i}$ are finite and $\lim_{t\rightarrow\infty}\frac{1}{\sqrt{t}}\sum_{u=1}^{t-1}\sum_{s=1+u}^{\infty}\alpha(s)=0$.
\begin{equation}
\frac{1}{\sqrt{t}}
\left(L_{t}-\frac{\nu\mathbb{E}[\ell_{1,1}] t}{1-\Vert\alpha\Vert_{1}\mathbb{E}[\ell_{1,1}]}\right)
\rightarrow N\left(0,\frac{\nu\mathbb{E}[\ell_{1,1}^{2}]}{(1-\Vert\alpha\Vert_{1}\mathbb{E}[\ell_{1,1}])^{3}}\right),
\end{equation}
in distribution as $t\rightarrow\infty$.
\end{theorem}

\begin{example}
To verify the central limit theorems (Theorem~\ref{CLT1} and Theorem~\ref{CLT2}), we run numerical simulations in Figure \ref{fig:CLT}. In the empirical experiment, we simulated 10,000 sample paths. And each sample path has 10,000 time steps. In our simulation, we set marks follow exponential distribution with inverse scale $\gamma = 0.3$, the base intensity $\nu = 0.1$ and decay rate $\alpha = 0.5$. The red line shows the normal distribution with theoretical mean and variance predicted by Theorem~\ref{CLT1} and Theorem~\ref{CLT2} and the blue bars represents the simulation results.
\begin{figure}[ht]
	\centering
	\subfloat[Histogram of $N_t$ samples and Normal distribution]{{\includegraphics[width=6.6cm]{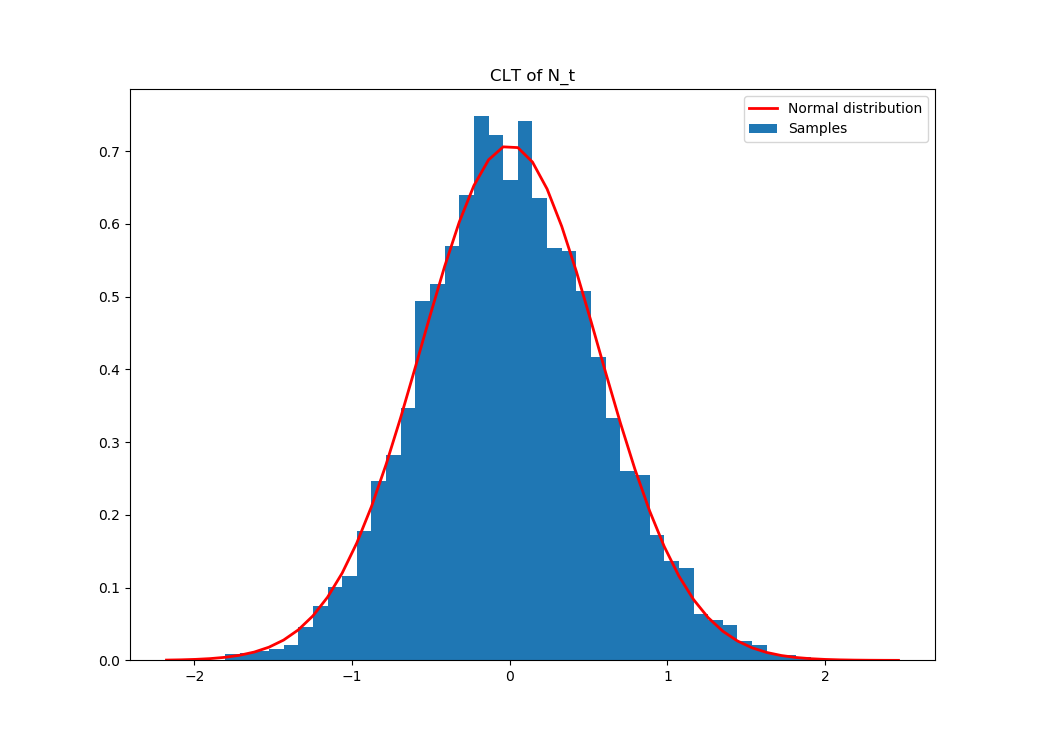}}}%
	\quad
	\subfloat[Histogram of $L_t$ samples and Normal distribution]{{\includegraphics[width=6.3cm]{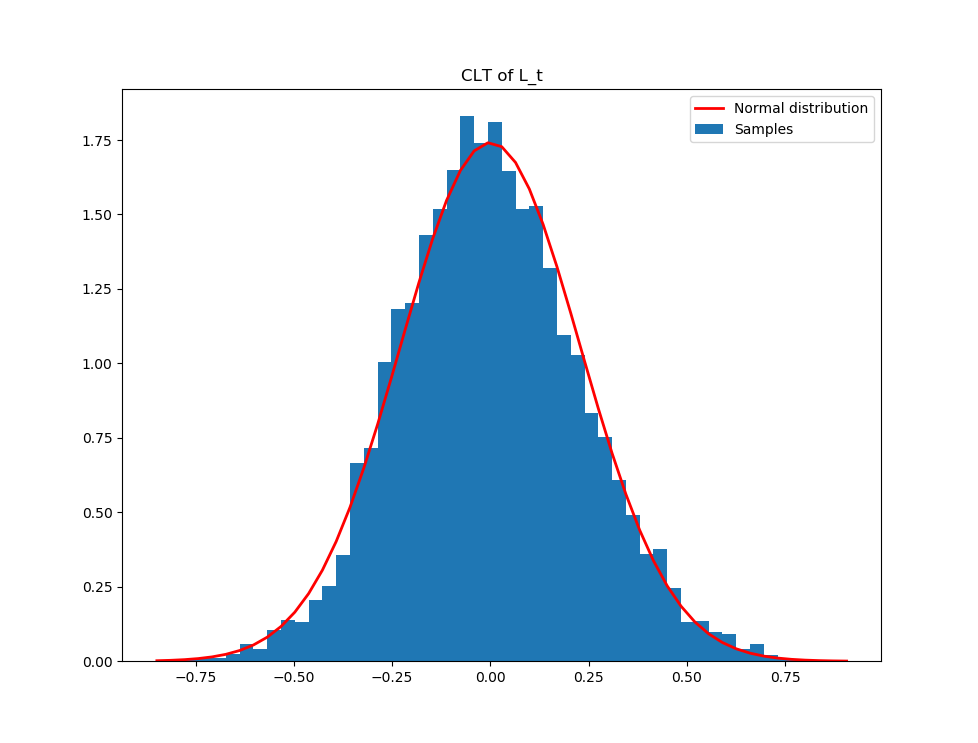} }}%

	\caption{Histogram of samples and Theoretical Normal distribution}
	\label{fig:CLT}
\end{figure}
\end{example}

\section{Proof of Main Results}

Before we proceed to the proof of Theorem~\ref{LLN1}, let us first
state and prove the following lemma.

\begin{lemma}\label{lem:main}
For any $t\in\hat{\mathbb{N}}$,
\begin{equation}
\mathbb{E}[Z_{t}]\leq\frac{\nu}{1-\Vert\alpha\Vert_{1}\mathbb{E}[\ell_{1,1}]},
\qquad
\mathbb{E}[X_{t}]\leq\frac{\nu\mathbb{E}[\ell_{1,1}]}{1-\Vert\alpha\Vert_{1}\mathbb{E}[\ell_{1,1}]}.
\end{equation}
\end{lemma}

\begin{proof}[Proof of Lemma \ref{lem:main}]
	Let us do an induction on $t$. 
	When $t=1$, $\mathbb{E}[Z_{1}]=\mathbb{E}[\lambda_{1}]=\nu<\frac{\nu}{1-\Vert\alpha\Vert_{1}\mathbb{E}[\ell_{1,1}]}$.
	Now, assume that $\mathbb{E}[Z_{s}]\leq\frac{\nu}{1-\Vert\alpha\Vert_{1}\mathbb{E}[\ell_{1,1}]}$ 
	for every $s=1,2,\ldots,t-1$.
	
	We can compute that
	\begin{equation}
	\mathbb{E}[Z_{t}]=\mathbb{E}[\lambda_{t}]
	=\nu+\sum_{s=1}^{t-1}\alpha(s)\mathbb{E}[X_{t-s}].
	\end{equation}
	Moreover, 
	\begin{equation}
	\mathbb{E}[X_{t-s}]=\mathbb{E}[\ell_{1,1}]\mathbb{E}[Z_{t-s}].
	\end{equation}
	Therefore, by induction,
	\begin{align*}
	\mathbb{E}[Z_{t}]
	&=\nu+\sum_{s=1}^{t-1}\alpha(s)\mathbb{E}[\ell_{1,1}]\mathbb{E}[Z_{t-s}]
	\\
	&\leq\nu+\sum_{s=1}^{t-1}\alpha(s)\mathbb{E}[\ell_{1,1}]\frac{\nu}{1-\Vert\alpha\Vert_{1}\mathbb{E}[\ell_{1,1}]}
	\\
	&\leq\nu+\Vert\alpha\Vert_{1}\mathbb{E}[\ell_{1,1}]\frac{\nu}{1-\Vert\alpha\Vert_{1}\mathbb{E}[\ell_{1,1}]}
	=\frac{\nu}{1-\Vert\alpha\Vert_{1}\mathbb{E}[\ell_{1,1}]}.
	\end{align*}
	Hence, we proved that for every $t\in\mathbb{N}$,
	$\mathbb{E}[Z_{t}]\leq\frac{\nu}{1-\Vert\alpha\Vert_{1}\mathbb{E}[\ell_{1,1}]}$.
	Since $\mathbb{E}[X_{t}]=\mathbb{E}[\ell_{1,1}]\mathbb{E}[Z_{t}]$,
	we conclude that for every $t\in\mathbb{N}$,
	$\mathbb{E}[X_{t}]\leq\frac{\nu\mathbb{E}[\ell_{1,1}]}{1-\Vert\alpha\Vert_{1}\mathbb{E}[\ell_{1,1}]}$.
	The proof is complete.
\end{proof}

\begin{proof}[Proof of Theorem \ref{LLN1}]
	First, by the definition, we can check that
	\begin{equation}
	\mathbb{E}\left[N_{t}-\sum_{s=1}^{t}\lambda_{s}|\mathcal{F}_{t-1}\right]
	=N_{t-1}-\sum_{s=1}^{t-1}\lambda_{s}
	+\mathbb{E}[Z_{t}-\lambda_{t}|\mathcal{F}_{t-1}]
	=N_{t-1}-\sum_{s=1}^{t-1}\lambda_{s},
	\end{equation}
	by the property of the Poisson random variable, where $\mathcal{F}_{t}$
	is the natural filtration up to time $t$.
	Based on this observation, we can readily see that
	$N_{t}-\sum_{s=1}^{t}\lambda_{s}$ is a martingale.
	
	Next, we can compute that
	\begin{align}
	\sum_{s=1}^{t}\lambda_{s}
	&=\nu t
	+\sum_{s=1}^{t}\sum_{u=1}^{s-1}\alpha(u)X_{s-u}
	\nonumber
	\\
	&=\nu t
	+\sum_{u=1}^{t-1}\sum_{s=u+1}^{t}\alpha(s-u)X_{u}
	\nonumber
	\\
	&=\nu t
	+\sum_{u=1}^{t-1}\sum_{s=u+1}^{\infty}\alpha(s-u)X_{u}
	-\sum_{u=1}^{t-1}\sum_{s=t+1}^{\infty}\alpha(s-u)X_{u}
	\nonumber
	\\
	&=\nu t
	+\Vert\alpha\Vert_{1}\sum_{u=1}^{t-1}X_{u}
	-\sum_{u=1}^{t-1}\sum_{s=t+1}^{\infty}\alpha(s-u)X_{u}
	\label{derive1}
	\\
	&=\nu t
	+\Vert\alpha\Vert_{1}\sum_{u=1}^{t-1}\mathbb{E}[\ell_{1,1}]Z_{u}
	+\mathcal{E}_{1}-\mathcal{E}_{2}
	\nonumber
	\\
	&=\nu t
	+\Vert\alpha\Vert_{1}\mathbb{E}[\ell_{1,1}]N_{t}
	-\mathcal{E}_{0}+\mathcal{E}_{1}-\mathcal{E}_{2},
	\nonumber
	\end{align}
	where
	\begin{align*}
	&\mathcal{E}_{0}:=\Vert\alpha\Vert_{1}\mathbb{E}[\ell_{1,1}]Z_{t},
	\\
	&\mathcal{E}_{1}:=\Vert\alpha\Vert_{1}\sum_{u=1}^{t-1}\left(X_{u}-\mathbb{E}[\ell_{1,1}]Z_{u}\right)
	\\
	&\mathcal{E}_{2}:=\sum_{u=1}^{t-1}\sum_{s=t+1}^{\infty}\alpha(s-u)X_{u}.
	\end{align*}
	Note that $\mathcal{E}_{0}\geq 0$ a.s. and by Lemma \ref{lem:main}, 
	$\mathbb{E}[\mathcal{E}_{0}]\leq\Vert\alpha\Vert_{1}\mathbb{E}[\ell_{1,1}]\frac{\nu}{1-\Vert\alpha\Vert_{1}\mathbb{E}[\ell_{1,1}]}$ so that by Chebychev's inequality, we get
	$\mathcal{E}_{0}/t\rightarrow 0$ in probability as $t\rightarrow\infty$.
	Moreover, we can compute that
	\begin{align*}
	\mathbb{E}[\mathcal{E}_{1}^{2}]
	&=\Vert\alpha\Vert_{1}^{2}\mathbb{E}\left[\left(\sum_{u=1}^{t-1}\sum_{j=1}^{Z_{u}}(\ell_{u,j}-\mathbb{E}[\ell_{1,1}])\right)^{2}\right]
	=\Vert\alpha\Vert_{1}^{2}\sum_{u=1}^{t-1}\text{Var}(\ell_{1,1})\mathbb{E}[Z_{u}]
	\\
	&\leq t\cdot\Vert\alpha\Vert_{1}^{2}\text{Var}(\ell_{1,1})\frac{\nu}{1-\Vert\alpha\Vert_{1}\mathbb{E}[\ell_{1,1}]},
	\end{align*}
	where we used Lemma \ref{lem:main}.
	By Chebychev's inequality, we get $\mathcal{E}_{1}/t\rightarrow 0$ in probability
	as $t\rightarrow\infty$.
	Next, $\mathcal{E}_{2}\geq 0$ a.s. and by Lemma \ref{lem:main},
	\begin{align*}
	\mathbb{E}[\mathcal{E}_{2}]
	&\leq\sum_{u=1}^{t-1}\sum_{s=t+1}^{\infty}\alpha(s-u)
	\frac{\nu\mathbb{E}[\ell_{1,1}]}{1-\Vert\alpha\Vert_{1}\mathbb{E}[\ell_{1,1}]}
	\\
	&=\left(\sum_{s=2}^{\infty}\alpha(s)+\sum_{s=3}^{\infty}\alpha(s)+\cdots+\sum_{s=t}^{\infty}\alpha(s)\right)\frac{\nu\mathbb{E}[\ell_{1,1}]}{1-\Vert\alpha\Vert_{1}\mathbb{E}[\ell_{1,1}]}.
	\end{align*}
	Since $\alpha$ is summable, $\lim_{t\rightarrow\infty}\sum_{s=t}^{\infty}\alpha(s)=0$,
	which implies that $\frac{1}{t}\mathbb{E}[\mathcal{E}_{2}]\rightarrow 0$,
	which yields that $\mathcal{E}_{2}/t\rightarrow 0$ in probability
	as $t\rightarrow\infty$.
	
	Finally, notice that
	\begin{equation}\label{derive2}
	N_{t}-\sum_{s=1}^{t}\lambda_{s}=(1-\Vert\alpha\Vert_{1}\mathbb{E}[\ell_{1,1}])N_{t}-\nu t
	+\mathcal{E}_{0}-\mathcal{E}_{1}+\mathcal{E}_{2},
	\end{equation}
	and $N_{t}-\sum_{s=1}^{t}\lambda_{s}$ is a martingale. Then with Lemma \ref{lem:main}, it is not hard to show
	\begin{align*}
	\mathbb{E}\left[\left(N_{t}-\sum_{s=1}^{t}\lambda_{s}\right)^{2}\right]
	&=\mathbb{E}\left[\sum_{s=1}^{t}\lambda_{s}\right]
	=\sum_{s=1}^{t}\left(\nu+\sum_{u=1}^{s-1}\alpha(u)\mathbb{E}X_{s-u}\right)
	\\
	&\leq
	t\left(\nu+\Vert\alpha\Vert_{1}\frac{\nu\mathbb{E}[\ell_{1,1}]}{1-\Vert\alpha\Vert_{1}\mathbb{E}[\ell_{1,1}]}\right).
	\end{align*}
	Therefore $(N_{t}-\sum_{s=1}^{t}\lambda_{s})/t\rightarrow 0$
	in probability as $t\rightarrow\infty$, which implies that
	\begin{equation}
	\frac{1}{t}\left((1-\Vert\alpha\Vert_{1}\mathbb{E}[\ell_{1,1}])N_{t}-\nu t\right)\rightarrow 0,
	\end{equation}
	in probability as $t\rightarrow\infty$. Thus, we can rearrange the terms to get 
	\begin{equation}
	\frac{N_t}{t} \to \frac{\nu}{1-\Vert\alpha\Vert_{1}\mathbb{E}[\ell_{1,1}]} ,
	\end{equation}
	in probability as $t\to\infty$, which completes the proof.
\end{proof}

\begin{proof}[Proof of Theorem \ref{LLN2}]
	We derived in \eqref{derive1} in the proof of Theorem \ref{LLN1} that
	\begin{equation}\label{rewrite}
	\sum_{s=1}^{t}\lambda_{s}=\nu t+\Vert\alpha\Vert_{1}\sum_{u=1}^{t-1}X_{u}-\mathcal{E}_{2},
	\end{equation}
	where $\mathcal{E}_{2}=\sum_{u=1}^{t-1}\sum_{s=t+1}^{\infty}\alpha(s-u)X_{u}$.
	We can rewrite \eqref{rewrite} as
	\begin{equation}
	N_{t}-\left(N_{t}-\sum_{s=1}^{t}\lambda_{s}\right)
	=\nu t+\Vert\alpha\Vert_{1}L_{t}
	-\Vert\alpha\Vert_{1}X_{t}-\mathcal{E}_{2},
	\end{equation}
	and by Lemma \ref{lem:main}, $\mathbb{E}[X_{t}]$ is uniformly bounded in $t$
	and thus $\Vert\alpha\Vert_{1}X_{t}/t\rightarrow 0$ in probability as $t\rightarrow\infty$.
	In the proof of Theorem \ref{LLN1}, we have shown that
	$\left(N_{t}-\sum_{s=1}^{t}\lambda_{s}\right)/t\rightarrow 0$ 
	in probability as $t\rightarrow\infty$, and $\mathcal{E}_{2}/t\rightarrow 0$
	in probability as $t\rightarrow\infty$, and 
	$N_{t}/t\rightarrow\frac{\nu}{1-\Vert\alpha\Vert_{1}\mathbb{E}[\ell_{1,1}]}$
	in probability as $t\rightarrow\infty$. Hence, 
	\begin{equation}
	\frac{1}{t}\left(\nu t+\Vert\alpha\Vert_{1}L_{t}\right)
	\rightarrow\frac{\nu}{1-\Vert\alpha\Vert_{1}\mathbb{E}[\ell_{1,1}]},
	\end{equation}
	in probability as $t\rightarrow\infty$.
	The proof is complete.
\end{proof}

Before we proceed to the proof of Theorem \ref{CLT1} and Theorem \ref{CLT2}, we first
prove a technical lemma.

\begin{lemma}\label{lem:lambda2}
For any $t\in\hat{\mathbb{N}}$,
\begin{equation}
\mathbb{E}[\lambda^2_{t}]\leq \dfrac{1}{1-\Vert\alpha\Vert^2_{1} E^2[\ell_{1,1}]}\left(\dfrac{\nu\Vert\alpha\Vert^2_{1}\text{Var}(\ell_{1,1})}{1-\Vert\alpha\Vert_{1}E[\ell_{1,1}]}+\dfrac{\nu^2(1+\Vert\alpha\Vert_{1} E[\ell_{1,1}])}{1-\Vert\alpha\Vert_{1} E[\ell_{1,1}]}\right).
\end{equation}
\end{lemma}

\begin{proof}[Proof of Lemma \ref{lem:lambda2}]
Throughout the paper, we assume that $\Vert\alpha\Vert_{1}\mathbb{E}[\ell_{1,1}]<1$. We prove the upper bound by induction. Assume 
$$E[\lambda^2_{t}]\le 
\dfrac{1}{1-\Vert\alpha\Vert^2_{1} E^2[\ell_{1,1}]}\left(\dfrac{\nu\Vert\alpha\Vert^2_{1}\text{Var}(\ell_{1,1})}{1-\Vert\alpha\Vert_{1}E[\ell_{1,1}]}+\dfrac{\nu^2(1+\Vert\alpha\Vert_{1} E[\ell_{1,1}])}{1-\Vert\alpha\Vert_{1} E[\ell_{1,1}]}\right).$$

First, if $t=1$, $E[\lambda^2_1]
\le \dfrac{\nu^2(1+\Vert\alpha\Vert_{1} E[\ell_{1,1}])}{1-\Vert\alpha\Vert_{1} E[\ell_{1,1}]}$. By the assumption of $\Vert\alpha\Vert_{1}\mathbb{E}[\ell_{1,1}]<1$, it is not hard to see
$E[\lambda^2_1] \le \dfrac{1}{1-\Vert\alpha\Vert^2_1 E^2[\ell_{1,1}]}\left(\dfrac{\nu\Vert\alpha\Vert^2_1\text{Var}(\ell_{1,1})}{1-\Vert\alpha\Vert_{1}E[\ell_{1,1}]}+\dfrac{\nu^2(1+\Vert\alpha\Vert_{1} E[\ell_{1,1}])}{1-\Vert\alpha\Vert_{1} E[\ell_{1,1}]}\right).$ Then, by induction,
\begin{align*}
    E[\lambda^2_{t+1}]
    &=  \nu^2 + 2\nu\sum^{t}_{s=1}\alpha(s)E[X_{t-s}] + E\left[\left(\sum^{t}_{s=1}\alpha(s)X_{t-s} \right)^2\right] \\
    &= \nu^2 + 2\nu\left(E[\lambda_t]-\nu\right) + E\left[\left(\sum^{t}_{s=1}\alpha^{1/2}(s)\alpha^{1/2}(s)X_{t-s} \right)^2\right]\\
    &\le -\nu^2 + 2\nu E[\lambda_t] + E\left[\left(\sum^{t}_{s=1}\alpha(s)\sum^{t}_{\tau=1}\alpha(\tau)X^2_{t-\tau} \right)\right] \\
    &= -\nu^2 + 2\nu E[\lambda_t] + \sum^{t-1}_{s=1}\alpha(s)\sum^{t}_{\tau=1}\alpha(t-\tau)E\left[X^2_{\tau}\right]\\
    &=-\nu^2 + 2\nu E[\lambda_t] + \sum^{t}_{s=1}\alpha(s)\sum^{t}_{\tau=1}\alpha(t-\tau)E\left(\text{Var}(X_\tau|\mathbf{F}_{\tau-1})+E^2(X_\tau|\mathbf{F}_{\tau-1})\right) \\
    &= -\nu^2 + 2\nu E[\lambda_t] + \sum^{t}_{s=1}\alpha(s)\sum^{t}_{\tau=1}\alpha(t-\tau)\left(\text{Var}(\ell_{1,1})E(\lambda_\tau)+E(\lambda^2_\tau)E^2(\ell_{1,1})\right)\\
    &= -\nu^2 +\left(2\nu+\sum^{t}_{s=1}\alpha(s)\sum^{t}_{\tau=1}\alpha(t-\tau)\text{Var}(\ell_{1,1})\right)E(\lambda_\tau)+\sum^{t}_{s=1}\alpha(s)\sum^{t}_{\tau=1}\alpha(t-\tau)E^2(\ell_{1,1})E(\lambda^2_\tau)\\
    &\le \dfrac{\nu^2(1+\Vert\alpha\Vert_{1} E[\ell_{1,1}])}{1-\Vert\alpha\Vert_{1} E[\ell_{1,1}]} + \Vert\alpha\Vert^2_{1}\text{Var}(\ell_{1,1})\dfrac{\nu}{1-\Vert\alpha\Vert_{1} E[\ell_{1,1}]} \\
    &+ \Vert\alpha\Vert^2_{1}E^2(\ell_{1,1})\dfrac{1}{1-\Vert\alpha\Vert^2_{1} E^2[\ell_{1,1}]}\left(\dfrac{\nu\Vert\alpha\Vert^2_{1}\text{Var}(\ell_{1,1})}{1-\Vert\alpha\Vert_{1}E[\ell_{1,1}]}+\dfrac{\nu^2(1+\Vert\alpha\Vert_{1} E[\ell_{1,1}])}{1-\Vert\alpha\Vert_{1} E[\ell_{1,1}]}\right) \\
    &= \dfrac{1}{1-\Vert\alpha\Vert^2_{1} E^2[\ell_{1,1}]}\left(\dfrac{\nu\Vert\alpha\Vert^2_{1}\text{Var}(\ell_{1,1})}{1-\Vert\alpha\Vert_{1}E[\ell_{1,1}]}+\dfrac{\nu^2(1+\Vert\alpha\Vert_{1} E[\ell_{1,1}])}{1-\Vert\alpha\Vert_{1} E[\ell_{1,1}]}\right).
\end{align*}
	The proof is complete.
\end{proof}

\begin{proof}[Proof of Theorem~\ref{CLT1}]
	We derived in \eqref{derive2} in the proof of Theorem \ref{LLN1} that
	\begin{equation}
	N_{t}-\sum_{s=1}^{t}\lambda_{s}=(1-\Vert\alpha\Vert_{1}\mathbb{E}[\ell_{1,1}])N_{t}-\nu t
	+\mathcal{E}_{0}-\mathcal{E}_{1}+\mathcal{E}_{2},
	\end{equation}
	and recall the definition of $\mathcal{E}_{1}$ from the proof of Theorem \ref{LLN1},
	we have $\mathcal{E}_{1}=\Vert\alpha\Vert_{1}\sum_{u=1}^{t-1}\left(X_{u}-\mathbb{E}[\ell_{1,1}]Z_{u}\right)$.
	Therefore, we have
	\begin{align}\label{derive3}
	&(1-\Vert\alpha\Vert_{1}\mathbb{E}[\ell_{1,1}])\left(N_{t}-\frac{\nu t}{1-\Vert\alpha\Vert_{1}\mathbb{E}[\ell_{1,1}]}\right)
	\\
	&=N_{t}-\sum_{s=1}^{t}\lambda_{s}
	+\Vert\alpha\Vert_{1}\sum_{u=1}^{t}\left(X_{u}-\mathbb{E}[\ell_{1,1}]Z_{u}\right)
	-\mathcal{E}_{0}-\mathcal{E}_{2}-\mathcal{E}_{3},
	\nonumber
	\end{align}
	where
	\begin{equation}
	\mathcal{E}_{3}:=\Vert\alpha\Vert_{1}\left(X_{t}-\mathbb{E}[\ell_{1,1}]Z_{t}\right).
	\end{equation}
	Note that by the proof of Theorem \ref{LLN1}, $\mathcal{E}_{0}\geq 0$ a.s. and 
	$\mathbb{E}[\mathcal{E}_{0}]$ is bounded in $t$
	and thus $\mathcal{E}_{0}/\sqrt{t}\rightarrow 0$
	in probability as $t\rightarrow\infty$.
	Moreover, $\mathbb{E}|\mathcal{E}_{3}|\leq\Vert\alpha\Vert_{1}\mathbb{E}[X_{t}]
	+\Vert\alpha\Vert_{1}\mathbb{E}[\ell_{1,1}]\mathbb{E}[Z_{t}]$
	and by Lemma \ref{lem:main}, both $\mathbb{E}[X_{t}]$ and $\mathbb{E}[Z_{t}]$
	are bounded in $t$ and then $\mathcal{E}_{3}/\sqrt{t}\rightarrow 0$
	in probability as $t\rightarrow\infty$.
	Next, by the proof of Theorem \ref{LLN1}, we have
	\begin{equation*}
	\mathbb{E}[\mathcal{E}_{2}]
	\leq\left(\sum_{s=2}^{\infty}\alpha(s)+\sum_{s=3}^{\infty}\alpha(s)+\cdots+\sum_{s=t}^{\infty}\alpha(s)\right)\frac{\nu\mathbb{E}[\ell_{1,1}]}{1-\Vert\alpha\Vert_{1}\mathbb{E}[\ell_{1,1}]},
	\end{equation*}
	and by the assumption $\frac{1}{\sqrt{t}}\sum_{u=1}^{t-1}\sum_{s=1+u}^{\infty}\alpha(s)\rightarrow 0$
	as $t\rightarrow\infty$, by Chebychev's inequality, 
	we get $\mathcal{E}_{2}/\sqrt{t}\rightarrow 0$
	in probability as $t\rightarrow\infty$. 
	Thus, it suffices to show that  $\dfrac{(1-\Vert\alpha\Vert_{1}\mathbb{E}[\ell_{1,1}])}{\sqrt{t}}\left(N_{t}-\frac{\nu t}{1-\Vert\alpha\Vert_{1}\mathbb{E}[\ell_{1,1}]}\right)
	$ converges to a Gaussian distribution in probability as $t\to \infty$. 
	
	Next, 
 	note that both $N_{t}-\sum_{s=1}^{t}\lambda_{s}$
 	and $\Vert\alpha\Vert_{1}\sum_{u=1}^{t}\left(X_{u}-\mathbb{E}[\ell_{1,1}]Z_{u}\right)$
 	are martingales and hence their sum $M_{t} = N_{t}-\sum_{s=1}^{t}\lambda_{s}+\Vert\alpha\Vert_{1}\sum_{u=1}^{t}\left(X_{u}-\mathbb{E}[\ell_{1,1}]Z_{u}\right)$ is a martingale, 
    and 
    \begin{align*}
        M_t &=  Z_t-\lambda_t + \sum^{t-1}_{s=1}\left(Z_s-\lambda_s\right)\\
        &+\Vert\alpha\Vert_1\left(X_t-\mathbb{E}\left[\ell_{1,1}\right]Z_t+\sum^{t-1}_{u=0}\left(X_u-\mathbb{E}\left[\ell_{1,1}\right]Z_u\right)\right).
    \end{align*}
    Thus,
    \begin{align*}
        M_t&=\sum^{Z_t}_{i=1}\left(1+\Vert\alpha\Vert_1\left(\ell_{1,i}-\mathbb{E}\left[\ell_{1,1}\right]\right)\right)-\lambda_t+\sum^{t-1}_{u=1}\left(Z_{u}-\lambda_u+\Vert\alpha\Vert_1\left(X_u-\mathbb{E}\left[\ell_{1,1}\right]Z_u\right)
        \right)\\
        &=\sum^{Z_t}_{i=1}\left(1+\Vert\alpha\Vert_1\left(\ell_{1,i}-\mathbb{E}\left[\ell_{1,1}\right]\right)\right)-\lambda_t+M_{t-1}.
    \end{align*}
    Thus, we can compute the quadratic variation of $M_t$ as follows.
    \begin{equation}
        \left<M\right>_t=\sum^{t}_{s=1}\mathbb{E}\left[D_s^2|\mathbf{F_{s-1}}\right]
    \end{equation}
    \begin{align*}
        \sum^t_{s=1}\mathbb{E}\left[D^2_s|\mathbf{F_{s-1}}\right]&=\sum^t_{s=1}\mathbb{E}\left[\left.\left(\sum^{Z_s}_{i=1}\left(1+\Vert\alpha\Vert_1\left(\ell_{1,i}-\mathbb{E}\left[\ell_{1,1}\right]\right)\right)-\lambda_s\right)^2\right|\mathbf{F_{s-1}}\right]\\
        &=\sum^t_{s=1}\left(\mathbb{E}\left[Z^2_s|\mathbf{F_{s-1}}\right]+\mathbb{E}\left[Z_s|\mathbf{F_{s-1}}\right]\Vert\alpha\Vert^2_1\text{var}\left(\ell_{1,1}\right)-\lambda^2_s\right)\\
        &=\sum^t_{s=1}\left(\lambda_s+\lambda^2_s+\lambda_s\Vert\alpha\Vert^2_1\text{var}\left(\ell_{1,1}\right)-\lambda^2_s\right)\\
        &=\sum^{t}_{s=1}\lambda_s\left(1+\Vert\alpha\Vert^2_1\text{var}\left(\ell_{1,1}\right)\right).
    \end{align*}
    
    Recall the proof of Theorem \ref{LLN1}, $\dfrac{\left<M\right>_t}{t}$ converge to $\frac{\nu(1+\Vert\alpha\Vert_{1}^{2}\text{Var}(\ell_{1,1}))}{(1-\Vert\alpha\Vert_{1}\mathbb{E}[\ell_{1,1}])}$ in probability as $t\to \infty$. 
    Furthermore, it follows from Lemma \ref{lem:clc1}
    \begin{equation}
        \dfrac{1}{t}\sum^t_{s=1}\mathbb{E}\left[D^2_s\mathbf{I}_{\left\{|D_s|\ge\epsilon \sqrt{t}\right\}}|\mathbf{F_{s-1}}\right]\to 0,
    \end{equation}
    in probability, which is conditional Lindeberg's condition in discrete-time Martingale central limit theorem.
    By martingale central limit theorem \cite{Brown1971,HH}, the conclusion follows.
\end{proof}
\begin{lemma}[]\label{lem:clc1} For any $\epsilon\in\mathbb{R}$ and $\mathcal{F}_{s}$
	is the natural filtration up to time $s$,
    \begin{equation}
        \dfrac{1}{t}\sum^t_{s=1}\mathbb{E}\left[D^2_s\mathbf{I}_{\left\{|D_s|\ge\epsilon \sqrt{t}\right\}}|\mathbf{F_{s-1}}\right]\to 0,
    \end{equation}
    in probability, where $D_s = \sum^{Z_s}_{i=1}\left(1+\Vert\alpha\Vert_1\left(\ell_{1,i}-\mathbb{E}\left[\ell_{1,1}\right]\right)\right)-\lambda_s$.
\end{lemma}
\begin{proof}[Proof of Lemma \ref{lem:clc1}]

\begin{align*}
    E[D_s^4|\mathbf{F}_{s-1}]&=E\left[\left.\left(\Vert\alpha\Vert_1\sum^{Z_s}_{i=1}\left(\ell_{1,i}-E[\ell_1]\right)+\left(Z_s-\lambda_s\right)\right)^4\right|\mathbf{F}_{s-1}\right]\\
    &=\sum^4_{k=0}\binom{4}{k}E\left[A^k_s B^{4-k}_s|\mathbf{F}_{s-1}\right],
\end{align*}
where $A_s = \Vert\alpha\Vert_1\sum^{Z_s}_{i=1}\left(\ell_{1,i}-E[\ell_1]\right)$ and $B_s = Z_s-\lambda_s$.

Then we can find the follows.
\begin{align*}
    E\left[B^{4}_s|\mathbf{F}_{s-1}\right]=E\left[(Z_s-\lambda_s)^{4}|\mathbf{F}_{s-1}\right]=3\lambda^2_s+\lambda_s.
\end{align*}
\begin{align*}
    E\left[A B^{3}_s|\mathbf{F}_{s-1}\right]&=\Vert\alpha\Vert_1 E\left[\sum^{Z_s}_{i=1}\left(\ell_{1,i}-E[\ell_1]\right)|\mathbf{F}_{s-1}\right]E\left[(Z_s-\lambda_s)^3|\mathbf{F}_{s-1}\right] \\
    &= \lambda_s\Vert\alpha\Vert_1 E\left[\ell_{1,i}-E[\ell_1]|\mathbf{F}_{s-1}\right]E\left[(Z_s-\lambda_s)^3|\mathbf{F}_{s-1}\right]=0
\end{align*}
\begin{align*}
    E\left[A^2_s B^{2}_s|\mathbf{F}_{s-1}\right] &= \Vert\alpha\Vert^{2}_1E\left[\left.\left(\sum^{Z_s}_{i=1}\left(\ell_{1,i}-E[\ell_1]\right)\right)^2\right|\mathbf{F}_{s-1}\right]E\left[(Z_s-\lambda_s)^2|\mathbf{F}_{s-1}\right] \\
    &= \lambda_s \Vert\alpha\Vert^{2}_1 E\left[\left.\left(\sum^{Z_s}_{i=1}\left(\ell_{1,i}-E[\ell_1]\right)\right)^2\right|\mathbf{F}_{s-1}\right] \\
    &= \lambda^2_s \Vert\alpha\Vert^2_1 \text{var}(\ell_1)
\end{align*}
\begin{align*}
    E\left[A^3_s B_s|\mathbf{F}_{s-1}\right] &= \Vert\alpha\Vert^{3}_1 E\left[\left.\left(\sum^{Z_s}_{i=1}\left(\ell_{1,i}-E[\ell_1]\right)\right)^3\right|\mathbf{F}_{s-1}\right]E\left[(Z_s-\lambda_s)|\mathbf{F}_{s-1}\right] \\
    &= 0.
\end{align*}
Let the characteristic function of $\ell_{t,i}-E\left[\ell_{1,1}\right]$ be $\Phi_{\ell}(\theta)$ where $\theta=\mathrm{i} \omega$ and $\omega \in\mathcal{R}$. Let $\text{Kurt}(\ell_{1,1})$ be the kurtosis of $\ell_{1,i}$. In order to find $E\left[A^4_s|\mathbf{F}_{s-1}\right]$, we use the characteristic function, $\Phi_{Y}(\theta)=\exp{\left(\lambda_s\left(\Phi_{\ell}(\theta)-1\right)\right)}$, for compound Poisson random variable $Y_s = \sum^{Z_s}_{i=1}\left(\ell_{1,i}-E[\ell_{1,1}]\right)$.
\begin{align*}
    \Phi^{\prime}_Y(\theta)&=\lambda_s\Phi^{\prime}_\ell(\theta)\exp{\left(\lambda_s\left(\Phi_\ell(\theta)-1\right)\right)}\\
    \Phi^{''}_Y(\theta)&=\left(\lambda_s\Phi^{''}_{\ell}(\theta)+\left(\lambda_s\Phi^{\prime}_{\ell}(\theta)\right)^2\right)\exp{\left(\lambda_s\left(\Phi_\ell(\theta)-1\right)\right)}\\
    \Phi^{(3)}_Y(\theta)&=\left(\lambda_s\Phi^{(3)}_{\ell}(\theta)+(\lambda^2_s+2\lambda_s)\Phi^{''}_\ell(\theta)\Phi^{'}_\ell(\theta)+(\lambda_s\Phi^{'}_\ell(\theta))^3\right)\exp{\left(\lambda_s\left(\Phi_\ell(\theta)-1\right)\right)}\\
    \Phi^{(4)}_Y(\theta)&=\left(\lambda_s\Phi^{(4)}_\ell(\theta)+4\lambda^2_s\Phi^{(3)}_\ell(\theta)\Phi^{'}_\ell(\theta)+(\lambda^2_s+2\lambda_s)\left(\Phi^{''}_\ell(\theta)\right)^2 \right)\exp{\left(\lambda_s\left(\Phi_\ell(\theta)-1\right)\right)}\\
    &+\left((2\lambda^2_s+4\lambda^3_s)\left(\Phi^{'}_\ell(\theta)\right)^2\Phi^{''}_\ell(\theta)+\left(\lambda_s\Phi^{'}_\ell(\theta)\right)^4\right)\exp{\left(\lambda_s\left(\Phi_\ell(\theta)-1\right)\right)}.
\end{align*}
For $\theta=0$, $\Phi^{'}_\ell(0)=0$, $\Phi_\ell(0)=1$, $\Phi^{''}_\ell(0)=\text{var}(\ell_{1,1})$ and $\Phi^{(4)}_\ell(0)=\text{Kurt}(\ell_{1,1})\text{var}^2(\ell_{1,1})$. Thus, $E\left[A^4_s|\mathbf{F}_{s-1}\right]=\Vert\alpha\Vert^{4}_1\Phi^{(4)}_Y(0)=\Vert\alpha\Vert^{4}_1\left(\lambda_s \text{Kurt}(\ell_{1,1})\text{var}^2(\ell_{1,1})+(\lambda^2_s+2\lambda_s)\text{Var}^2(\ell_{1,1})\right)$. Then we can find the following,
\begin{align*}
    E\left[D^4_s|\mathbf{F}_{s-1}\right] = \lambda_s\left(1+\Vert\alpha\Vert^4_1\text{Var}^2(\ell_{1,1})(\text{Kurt}(\ell_{1,1})+2)\right) + \lambda^2_s\left(3+6\Vert\alpha\Vert^2_1\text{Var}(\ell_{1,1})+\Vert\alpha\Vert^4_1\text{Var}^2(\ell_{1,1})\right).
\end{align*}

Our goal is to show 
    \begin{equation}
        \dfrac{1}{t}\sum^t_{s=1}\mathbb{E}\left[D^2_s\mathbf{I}_{\left\{|D_s|\ge\epsilon \sqrt{t}\right\}}|\mathbf{F_{s-1}}\right]\to 0
    \end{equation}
in probability. As we know,
\begin{align*}
    \dfrac{1}{t}\sum^t_{s=1}\mathbb{E}\left[D^2_s\mathbf{I}_{\left\{|D_s|\ge\epsilon \sqrt{t}\right\}}|\mathbf{F_{s-1}}\right] 
    &\le \dfrac{1}{t}\sum^t_{s=1}\mathbb{E}\left[D^2_s\dfrac{D^2_s}{\epsilon^2 t}|\mathbf{F_{s-1}}\right] \\
    &\le \dfrac{1}{\epsilon^2 t^2}\sum^t_{s=1}\mathbb{E}\left[D^4_s|\mathbf{F_{s-1}}\right] \\
    &=\dfrac{1}{\epsilon^2 t^2}\sum^t_{s=1}\left(\lambda_s C_1 + \lambda^2_s C_2\right),
\end{align*}
where $C_1=1+\Vert\alpha\Vert^4_1\text{Var}^2(\ell_{1,1})(\text{Kurt}(\ell_{1,1})+2)$ and $C_2=3+6\Vert\alpha\Vert^2_1\text{Var}(\ell_{1,1})+\Vert\alpha\Vert^4_1\text{Var}^2(\ell_{1,1})$.

Because $\mathbb{E}\left[\lambda_s\right]<\infty$ and $\mathbb{E}\left[\lambda^2_s\right]<\infty$ by Lemma \ref{lem:main} and Lemma \ref{lem:lambda2}, as $t \to \infty$,  $\dfrac{1}{t}\sum^t_{s=1}\mathbb{E}\left[D^2_s\mathbf{I}_{\left\{|D_s|\ge\epsilon \sqrt{t}\right\}}|\mathbf{F_{s-1}}\right]  \to 0$ in probability by Markov inequality. 
\end{proof}

\begin{proof}[Proof of Theorem~\ref{CLT2}]
	We derived in the proof of Theorem \ref{CLT1} that
	\begin{align*}
	&(1-\Vert\alpha\Vert_{1}\mathbb{E}[\ell_{1,1}])\left(N_{t}-\frac{\nu t}{1-\Vert\alpha\Vert_{1}\mathbb{E}[\ell_{1,1}]}\right)
	\\
	&=N_{t}-\sum_{s=1}^{t}\lambda_{s}
	+\Vert\alpha\Vert_{1}\sum_{u=1}^{t}\left(X_{u}-\mathbb{E}[\ell_{1,1}]Z_{u}\right)
	-\mathcal{E}_{0}-\mathcal{E}_{2}-\mathcal{E}_{3},
	\end{align*}
	where $\mathcal{E}_{0}/\sqrt{t}, \mathcal{E}_{2}/\sqrt{t}, \mathcal{E}_{3}/\sqrt{t}\rightarrow 0$,
	in probability as $t\rightarrow\infty$. 
	Recall that $L_{t}=\sum_{s=1}^{t}X_{s}$ and $N_{t}=\sum_{s=1}^{t}Z_{s}$.
	It follows that
	\begin{align}
	&\Vert\alpha\Vert_{1}\left(L_{t}-\frac{\nu\mathbb{E}[\ell_{1,1}]t}{1-\Vert\alpha\Vert_{1}\mathbb{E}[\ell_{1,1}]}\right)
	\nonumber
	\\
	&=\sum_{s=1}^{t}\lambda_{s}-\frac{\nu t}{1-\Vert\alpha\Vert_{1}\mathbb{E}[\ell_{1,1}]}+\mathcal{E}_{0}+\mathcal{E}_{2}+\mathcal{E}_{3}
	\nonumber
	\\
	&=N_{t}-\frac{\nu t}{1-\Vert\alpha\Vert_{1}\mathbb{E}[\ell_{1,1}]}
	-\left(N_{t}-\sum_{s=1}^{t}\lambda_{s}\right)
	+\mathcal{E}_{0}+\mathcal{E}_{2}+\mathcal{E}_{3}
	\nonumber
	\\
	&=\frac{N_{t}-\sum_{s=1}^{t}\lambda_{s}
		+\Vert\alpha\Vert_{1}\sum_{u=1}^{t}\left(X_{u}-\mathbb{E}[\ell_{1,1}]Z_{u}\right)
		-\mathcal{E}_{0}-\mathcal{E}_{2}-\mathcal{E}_{3}}{1-\Vert\alpha\Vert_{1}\mathbb{E}[\ell_{1,1}]}
	\nonumber
	\\
	&\qquad\qquad\qquad\qquad
	-\left(N_{t}-\sum_{s=1}^{t}\lambda_{s}\right)
	+\mathcal{E}_{0}+\mathcal{E}_{2}+\mathcal{E}_{3}
	\nonumber
	\\
	&=\frac{\Vert\alpha\Vert_{1}\mathbb{E}[\ell_{1,1}](N_{t}-\sum_{s=1}^{t}\lambda_{s})
		+\Vert\alpha\Vert_{1}\sum_{u=1}^{t}\left(X_{u}-\mathbb{E}[\ell_{1,1}]Z_{u}\right)}
	{1-\Vert\alpha\Vert_{1}\mathbb{E}[\ell_{1,1}]}
	\nonumber
	\\
	&\qquad\qquad\qquad\qquad
	-\frac{\mathcal{E}_{0}+\mathcal{E}_{2}+\mathcal{E}_{3}}{1-\Vert\alpha\Vert_{1}\mathbb{E}[\ell_{1,1}]}+\mathcal{E}_{0}+\mathcal{E}_{2}+\mathcal{E}_{3},\label{derive4}
	\end{align}
	where we used \eqref{derive3}.
	Similar as in the proof of Theorem \ref{CLT1}, 
	we can observe that $\Vert\alpha\Vert_{1}\mathbb{E}[\ell_{1,1}](N_{t}-\sum_{s=1}^{t}\lambda_{s})
	+\Vert\alpha\Vert_{1}\sum_{s=1}^{t}\left(X_{s}-\mathbb{E}[\ell_{1,1}]Z_{s}\right)$
	is a martingale with 
	$$\left<M\right>_t=\sum^{t}_{s}\mathbb{E}\left[\Tilde{D}^2_s|\mathbf{F_{s-1}}\right],$$
	where
	$$\Tilde{D}^2_t=\Vert\alpha\Vert^2_1\left(\mathbb{E}[\ell_{1,1}]\left(Z_t-\lambda_t\right)+\sum^{Z_t}_{i=1}\left(\ell_{i,1}-\mathbb{E}[\ell_{1,1}]\right)\right)^2.$$
Then we have
  \begin{align*}
  \sum^{t}_{s=1}\mathbb{E}[\Tilde{D}^2_s|\mathbf{F_{s-1}}]&=\Vert\alpha\Vert^2_1\sum^{t}_{s=1}\left(\lambda_s \mathbb{E}[\ell_{1,1}]^2+\mathbb{E}\left[\left.\left(\sum^{Z_s}_{i=1}\left(\ell_{i,j}-\mathbb{E}[\ell_{1,1}]\right)\right)^2\right|\mathbf{F}_{s-1}\right]\right)\\
  &=\Vert\alpha\Vert^2_1\sum^{t}_{s=1}\lambda_s\mathbb{E}[\ell^2_{1,1}].
  \end{align*}
  As $t\to\infty$, $\sum^{t}_{s=1}\mathbb{E}[\Tilde{D}^2_s|\mathbf{F_{s-1}}]\to\frac{\Vert\alpha\Vert^2_1\nu\mathbb{E}[\ell_{1,1}^{2}]}{(1-\Vert\alpha\Vert_{1}\mathbb{E}[\ell_{1,1}])}$ in probability. Furthermore, we have Lemma \ref{lem:clc2}, which is conditional Lindeberg's condition in discrete-time Martingale central limit theorem.
  Then by martingale central limit theorem, the conclusion follows.
\end{proof}

\begin{lemma}[]\label{lem:clc2} For any $\epsilon \in\mathbb{R}$ and $\mathcal{F}_{s}$
	is the natural filtration up to time $s$,
    \begin{equation}
        \dfrac{1}{t}\sum^t_{s=1}\mathbb{E}\left[\Tilde{D}^2_s\mathbf{I}_{\left\{|\Tilde{D}_s|\ge\epsilon \sqrt{t}\right\}}|\mathbf{F_{s-1}}\right]  \to 0,
    \end{equation}
    in probability, where $\Tilde{D}_t=\Vert\alpha\Vert_1\left(\mathbb{E}[\ell_{1,1}]\left(Z_t-\lambda_t\right)+\sum^{Z_t}_{i=1}\left(\ell_{i,1}-\mathbb{E}[\ell_{1,1}]\right)\right)$.
\end{lemma}

\begin{proof}[Proof of Lemma \ref{lem:clc2}]
	
\begin{align*}
  \sum^{t}_{s=1}\mathbb{E}[\Tilde{D}^4_s|\mathbf{F_{s-1}}]&=\Vert\alpha\Vert^4_1\sum^{t}_{s=1}\mathbb{E}\left[\left.\left(\mathbb{E}[\ell_{1,1}]\left(Z_t-\lambda_t\right)+\sum^{Z_t}_{i=1}\left(\ell_{i,1}-\mathbb{E}[\ell_{1,1}]\right)\right)^4\right|\mathbf{F}_{s-1}\right]\\
  &=\Vert\alpha\Vert^4_1\sum^4_{k=0}\binom{4}{k}E\left[A^k_s B^{4-k}_s|\mathbf{F}_{s-1}\right],
\end{align*}
 
where $A_s = \sum^{Z_s}_{i=1}\left(\ell_{1,i}-E[\ell_{1,1}]\right)$ and $B_s =\mathbb{E}[\ell_{1,1}]\left( Z_s-\lambda_s\right)$.
Then we can find the follows.
\begin{align*}
    E\left[B^{4}_s|\mathbf{F}_{s-1}\right]=\mathbb{E}^4[\ell_{1,1}]E\left[(Z_s-\lambda_s)^{4}|\mathbf{F}_{s-1}\right]=\mathbb{E}^4[\ell_{1,1}]\left(3\lambda^2_s+\lambda_s\right)
\end{align*}

\begin{align*}
    E\left[A B^{3}_s|\mathbf{F}_{s-1}\right]&=\mathbb{E}^3[\ell_{1,1}] E\left[\sum^{Z_s}_{i=1}\left(\ell_{1,i}-E[\ell_{1,1}]\right)|\mathbf{F}_{s-1}\right]E\left[(Z_s-\lambda_s)^3|\mathbf{F}_{s-1}\right] \\
    &= \lambda_s\mathbb{E}^3[\ell_{1,1}] E\left[\ell_{1,i}-E[\ell_{1,1}]|\mathbf{F}_{s-1}\right]E\left[(Z_s-\lambda_s)^3|\mathbf{F}_{s-1}\right]=0
\end{align*}
\begin{align*}
    E\left[A^2_s B^{2}_s|\mathbf{F}_{s-1}\right] &= \mathbb{E}^2[\ell_{1,1}]E\left[\left.\left(\sum^{Z_s}_{i=1}\left(\ell_{1,i}-E[\ell_{1,1}]\right)\right)^2\right|\mathbf{F}_{s-1}\right]E\left[(Z_s-\lambda_s)^2|\mathbf{F}_{s-1}\right] \\
    &= \lambda_s \mathbb{E}^2[\ell_{1,1}] E\left[\left.\left(\sum^{Z_s}_{i=1}\left(\ell_{1,i}-E[\ell_{1,1}]\right)\right)^2\right|\mathbf{F}_{s-1}\right] \\
    &= \lambda^2_s \mathbb{E}^2[\ell_{1,1}] \text{var}(\ell_{1,1})
\end{align*}
\begin{align*}
    E\left[A^3_s B_s|\mathbf{F}_{s-1}\right] &= \mathbb{E}[\ell_{1,1}] E\left[\left.\left(\sum^{Z_s}_{i=1}\left(\ell_{1,i}-E[\ell_{1,1}]\right)\right)^3\right|\mathbf{F}_{s-1}\right]E\left[(Z_s-\lambda_s)|\mathbf{F}_{s-1}\right] \\
    &= 0.
\end{align*}
Let the characteristic function of $\ell_{1,i}-E\left[\ell_{1,1}\right]$ be $\Phi_{\ell}(\theta)$ where $\theta=\mathrm{i} \omega$ and $\omega \in\mathcal{R}$. In order to find $E\left[A^4_s|\mathbf{F}_{s-1}\right]$, we use the characteristic function, $\Phi_{Y}(\theta)=\exp{\left(\lambda_s\left(\Phi_{\ell}(\theta)-1\right)\right)}$, for compound Poisson random variable $Y_s = \sum^{Z_s}_{i=1}\left(\ell_{1,i}-E[\ell_{1,1}]\right)$.
\begin{align*}
    \Phi^{\prime}_Y(\theta)&=\lambda_s\Phi^{\prime}_\ell(\theta)\exp{\left(\lambda_s\left(\Phi_\ell(\theta)-1\right)\right)}\\
    \Phi^{''}_Y(\theta)&=\left(\lambda_s\Phi^{''}_{\ell}(\theta)+\left(\lambda_s\Phi^{\prime}_{\ell}(\theta)\right)^2\right)\exp{\left(\lambda_s\left(\Phi_\ell(\theta)-1\right)\right)}\\
    \Phi^{(3)}_Y(\theta)&=\left(\lambda_s\Phi^{(3)}_{\ell}(\theta)+(\lambda^2_s+2\lambda_s)\Phi^{''}_\ell(\theta)\Phi^{'}_\ell(\theta)+(\lambda_s\Phi^{'}_\ell(\theta))^3\right)\exp{\left(\lambda_s\left(\Phi_\ell(\theta)-1\right)\right)}\\
    \Phi^{(4)}_Y(\theta)&=\left(\lambda_s\Phi^{(4)}_\ell(\theta)+4\lambda^2_s\Phi^{(3)}_\ell(\theta)\Phi^{'}_\ell(\theta)+(\lambda^2_s+2\lambda_s)\left(\Phi^{''}_\ell(\theta)\right)^2 \right)\exp{\left(\lambda_s\left(\Phi_\ell(\theta)-1\right)\right)}\\
    &+\left((2\lambda^2_s+4\lambda^3_s)\left(\Phi^{'}_\ell(\theta)\right)^2\Phi^{''}_\ell(\theta)+\left(\lambda_s\Phi^{'}_\ell(\theta)\right)^4\right)\exp{\left(\lambda_s\left(\Phi_\ell(\theta)-1\right)\right)}.
\end{align*}
For $\theta=0$, $\Phi^{'}_\ell(0)=0$, $\Phi_\ell(0)=1$, $\Phi^{''}_\ell(0)=\text{var}(\ell_{1,1})$ and $\Phi^{(4)}_\ell(0)=\text{Kurt}(\ell_{1,1})\text{var}^2(\ell_{1,1})$. Thus, $E\left[A^4_s|\mathbf{F}_{s-1}\right]=\Phi^{(4)}_Y(0)=\left(\lambda_s \text{Kurt}(\ell_{1,1})\text{var}^2(\ell_{1,1})+(\lambda^2_s+2\lambda_s)\text{Var}^2(\ell_{1,1})\right)$. Then we can find the follows.
\begin{align*}
    E\left[\Tilde{D}^4_s|\mathbf{F}_{s-1}\right] = \lambda_s\left(\mathbb{E}^4[\ell_{1,1}]+\text{Var}^2(\ell_{1,1})(\text{Kurt}(\ell_{1,1})+2)\right) + \lambda^2_s\left(3\mathbb{E}^4[\ell_{1,1}]+6\text{Var}(\ell_{1,1})+\text{Var}^2(\ell_{1,1})\right)
\end{align*}
Similarly as the previous proof, we have
\begin{align*}
    \dfrac{1}{t}\sum^t_{s=1}\mathbb{E}\left[\Tilde{D}^2_s\mathbf{I}_{\left\{|\Tilde{D}_s|\ge\epsilon \sqrt{t}\right\}}|\mathbf{F_{s-1}}\right] 
    &\le \dfrac{1}{t}\sum^t_{s=1}\mathbb{E}\left[\left .\Tilde{D}^2_s\dfrac{\Tilde{D}^2_s}{\epsilon^2 t}\right|\mathbf{F_{s-1}}\right] \\
    &\le \dfrac{1}{\epsilon^2 t^2}\sum^t_{s=1}\mathbb{E}\left[\Tilde{D}^4_s|\mathbf{F_{s-1}}\right] \\
    &=\dfrac{1}{\epsilon^2 t^2}\sum^t_{s=1}\left(\lambda_s C_1 + \lambda^2_s C_2\right),
\end{align*}
where $C_1=\mathbb{E}^4[\ell_{1,1}]+\text{Var}^2(\ell_{1,1})(\text{Kurt}(\ell_{1,1})+2)$ and $C_2=3\mathbb{E}^4[\ell_{1,1}]+6\text{Var}(\ell_{1,1})+\text{Var}^2(\ell_{1,1})$.

Because $\mathbb{E}\left[\lambda_s\right]<\infty$ and $\mathbb{E}\left[\lambda^2_s\right]<\infty$ by Lemma \ref{lem:main} and Lemma \ref{lem:lambda2}, as $t \to \infty$,  $\dfrac{1}{t}\sum^t_{s=1}\mathbb{E}\left[\Tilde{D}^2_s\mathbf{I}_{\left\{|\Tilde{D}_s|\ge\epsilon \sqrt{t}\right\}}|\mathbf{F_{s-1}}\right]  \to 0$ in probability. 

\end{proof}

\newpage
\bibliographystyle{alpha}
\bibliography{mybibfile}

\end{document}